\documentclass[12pt]{article}

\usepackage{amsmath, amssymb, amsthm, bbm, mathtools}
\newtheorem{thm}{Theorem}
\newtheorem{defin}[thm]{Definition}
\newtheorem{lem}[thm]{Lemma}

\newtheorem{prop}[thm]{Proposition}

\usepackage{tcolorbox}
\usepackage{import}
\usepackage{pdfpages}
\usepackage{transparent}
\usepackage{xcolor}
\usepackage{lscape}
\usepackage{fullpage}
\usepackage{hyperref}
\usepackage{tikz}

\DeclareMathOperator{\term}{term}
\DeclareMathOperator{\init}{init}

\newtcolorbox[auto counter]{problem}[2][]{ colback=red!5!white,colframe=black!75!white,fonttitle=\bfseries, title=Problem~\thetcbcounter: #2,#1}

\title{Effective drift estimates for random walks on graph products}
\author{Kunal Chawla}
\date{}
\begin{document}
\maketitle
\begin{abstract} 
We find uniform lower bounds on the drift for a large family of random walks on graph products, of the form $ \mathbb{P} (|Z _{n}| \leq \kappa n) \leq e ^{-\kappa n}  $ for $ \kappa > 0 $. This includes the simple random walk for a right-angled Artin group with a sparse defining graph. This is done by extending an argument of Gou\"{e}zel, along with the combinatorial notion of a piling introduced by Crisp, Godelle, and Wiest. We do not use any moment conditions, instead considering random walks which alternate between one measure uniformly distributed on vertex groups, and another measure over which we make almost no assumptions. 	  \end{abstract}
 
 \section{Introduction}
 Let $ G $ be a group acting on a metric space $ X $, and let $ \mu $ be a probability measure. Let $ g _{1}, g _{2},\dots  $ be i.i.d. $ G $-valued random variables with distribution $ \mu $. One can construct a random walk on $ X $ by picking a basepoint $ o \in X $ and letting  
 \[ Z _{n}\cdot o = g _{1} \dots g _{n} \cdot o.\] 
 Often considered in the literature is qualitative long-term behaviour of $ Z _{n} $. Furstenberg showed that random walks on semi-simple Lie groups converge almost surely to a point on a natural boundary at infinity \cite{furstenberg1963poisson}. Kaimanovich identified the Poisson boundary for a general class of groups with hyperbolic properties \cite{kaimanovich2000poisson}. Karlsson and Margulis showed that certain random walks on Busemann non-positively curved spaces sublinearly track a geodesic \cite{karlsson1999multiplicative}, and Tiozzo exhibited a general condition to ensure sublinear tracking \cite{tiozzo2015sublinear}. Benoist and Quint \cite{benoist2016central} exhibited a central limit theorem for random walks with finite variance on Gromov hyperbolic groups. Maher-Tiozzo showed that a non-elementary random walk on a (not necessarily proper) hyperbolic space converges to the boundary \cite{maher2018random}. Nevo and Sageev identified the Poisson boundary for groups acting on CAT(0) cube complexes \cite{nevo2013poisson}. Most of these results rely on geometric assumptions about the group, usually some sort of nonpositive curvature condition, as well as moment or entropy assumptions on $ \mu $.

 In the recent literature are inquiries into large deviations principles for random walks on hyperbolic spaces. Let $ X $ be a Gromov hyperbolic $ G $-space with a basepoint $ o $. Maher and Tiozzo showed that if $ \mu $ has finite support, then $ \mathbb{P}(d(Z _{n}o, o) \leq \kappa n) $ decays exponentially for some $ \kappa $. This was upgraded to an exponential moment condition by \cite{sunderland2020linear}. Later, in \cite{boulanger2020large} it was shown that this statement holds for all $ \kappa $ up to the rate of escape \[\ell := \lim_{n \to \infty} \frac{\mathbb{E} \left[d(Z _{n} o, o)\right]}{n}.\] 

 Recently, Gou\"{e}zel \cite{gouezel2021} has shown, with a clever geometric argument, that all moment assumptions can be removed. This argument does not rely on boundary theory, and is entirely quantitative. The idea is that one can decompose a sample path into segments which go in one direction, and points where the path `pivots' from one direction to the next. By hyperbolicity, in most directions the sample path will move further away from the basepoint. One can bound the number of pivotal points from below by a sum of i.i.d. random variables with positive expectation. Then the result follows from the theory of large deviations. This argument has recently featured in two articles of Choi, to explore genericity of pseudo-Anosovs  \cite{choi2021pseudo} as well as sublinear tracking and central limit theorems \cite{choi2021central}.

We apply this technique to give effective estimates for the drift of certain random walks on graph products. Giving effective lower bounds for drift is a notoriously hard problem. Furstenberg exhibited an integral formula for the drift, however this is not effective because it requires knowledge about the harmonic measure. In this paper, we consider a class of random walks on graph products of groups acting on their Cayley graphs, which are not usually hyperbolic. Let $ \Gamma $ be a graph, with vertex set $V$ and edge set $E$. Suppose that for each vertex $ v \in V$, there is an assigned group $ G _{v} $, which has the (not necessarily finite) presentation $G_v = \langle S _{v}| R _{v}  \rangle $. The graph product, denoted by $G =  G(\Gamma) $, is the group defined by 
 \[ G(\Gamma) = \langle \sqcup _{v \in V} S _{v}| \sqcup _{v \in V} R _{v} \sqcup_{(v, w) \in E} [S _{v}, S _{w}]  \rangle  .\] 

For example, if the graph is a clique, then $ G $ is the direct product $ G _{1} \times \dots \times G _{n} $. If the graph has no edges, then $ G $ is the free product $ G _{1} * \dots * G _{n} $. If the graph is a path with 3 vertices and $ G _{1} = G _{2} = G _{3} = \mathbb{Z}  $, then $ G = \mathbb{Z} ^{2} * \mathbb{Z} $. Graph products need not be finitely generated or hyperbolic, for example if each vertex group is an infinite direct sum of copies of $ \mathbb{Z} $. In general, if each $ G _{i} $ is a copy of $ \mathbb{Z} $, then $ G $ is the right-angled Artin group on the graph in question. The graph product interpolates between the direct product and free product, where a sparse graph means that $ G $ is closer to a free product. 

 Given a graph $ \Gamma $, let $ D $ be the number of vertices, $ C $ the maximum size of a clique, and $ B $ the maximum size of a 1-neighbourhood of a clique. For example, if $ \Gamma $ is a cycle of length $ 5 $, then $ C = 2 $ and $ B = 4 $. We say that $ \Gamma $ has \emph{small cliques} if $ D > 3B+2C $. Moreover, we say that a measure on $ G $ is \emph{alternating} if it is of the form $ \mu * \nu $ where $ \mu(G _{v} \setminus \{ e \}) = 1/D $ for any $v \in V$, and $\nu$ is any probability measure with $ \nu(e) = 0 $. For example, if $ \mu $ is the measure driving the simple random walk, then $ \mu ^{2} $ is alternating.

In this article, we prove the following:
\begin{thm}
	Let $ \Gamma $ be a graph with small cliques and let $ G(\Gamma) $ be a graph product with vertex groups $ G _{1}, \dots, G _{D} $. 
	Then there exists an effective constant $ \kappa = \kappa(\Gamma) $ such that for any random walk $ (Z _{n}) $ driven by an alternating measure, we have 
	\[ \mathbb{P} (|Z _{n}| \leq \kappa n) \leq e^{-\kappa n} \qquad \textup{for any }n.\]
	\end{thm}

	We will see that $ \kappa $ depends only on $ B, C, $ and $ D $, and can be effectively computed (see page $ 8 $).
	
\textbf{Example.} For example, if $ \Gamma $ is a $ D $-cycle for $ D > 16 $ then $C = 2, B = 4 $, and $\kappa \geq 0.3$. 
 In particular, the drift for the simple random walk on the corresponding RAAG is at least $ 0.15 $. In this regime, we will see that $ \kappa \to 1$ as $ D \to \infty $. 

	Here, the word length $ |Z _{n}| $ is taken with respect to any generating set where each generator lies inside one of the vertex groups. As our argument will show, this estimate can be phrased in terms of the syllable length of $ Z _{n} $. That is to say, the minimum length of a representation $ Z _{n} = g _{i _{1}}\dots g _{i _{n}} $ where $ g _{i _{k}}, g _{i _{k+1}}  $ are elements of distinct vertex groups for all $ 1 \leq k \leq n-1 $. As a result, this $ \kappa $ is completely independent of the choice of vertex groups and the measure $ \nu $, and only depends on the graph structure. In the case that all the vertex groups are $ \mathbb{Z} $, then this gives us a quantitative drift estimate for certain right-angled Artin groups.	
	
	This gives a quantitative sense in which these RAAGs are closer to a free group than a free abelian group. The use of an alternating random walk $ \mu * \nu $ is notable because the measure $ \nu $ is allowed to have arbitrarily fat tails. All of the regularity comes from the $ \mu $. This choice of $ \kappa $ is then uniform over a large class of random walks. Instead of relying on moment conditions, we combine some ideas from Gou\"{e}zel's argument with a combinatorial tool. In particular, we extend the notion of a `piling' from \cite{crisp} to bring the notion of pivotal points over to the graph product setting.

		  To prove our main theorem, we write $ g _{i} = s _{i}w _{i} $ where $ s _{i} \sim \mu $ and $ w _{i} \sim \nu $. Then we fix the $ w _{i} $'s and keep the randomness coming from the $ s _{i} $'s. To each pair $ w _{n-1}, w _{n} $, we can find a $ s _{n} $ from a vertex far away on the graph, so that $ |Z _{n}| $ strictly increases. As the graph is sparse, there are many such choices for $ s _{n} $. Hence we can bound the length $ |Z _{n}| $ from below by a sum of $ n $ i.i.d. copies of some random variable $ U  $ with positive expectation.

		  The paper is organized as follows: in section 2 we define pilings and introduce the notion of a terminal and initial clique for elements of graph products. In section 3 we describe the notion of pivotal points, inspired by Gou\"{e}zel. Finally, in section 4, we prove our main theorem and state a formula for the drift. 
	 
\textbf{Acknowledgements:} The author would like to thank Giulio Tiozzo for many helpful conversations and comments. This work was completed under funding from an NSERC USRA.
\section{Pilings}
The notion of a piling was introduced by Crisp, Godelle, and Wiest in \cite{crisp} to give a normal form for right-angled Artin groups. They used pilings to solve the conjugacy problem in this setting. 

Since right-angled Artin groups interpolate between free groups and free abelian groups, one looks for a way to quantify how close a RAAG is to either extreme. One way to do this is to explore when a word in a RAAG locally looks like a word in a free product. Consider for example the group $ \mathbb{Z} ^{2}* \mathbb{Z} = \langle a,b,c | [a,b]  \rangle $. Then the word $ aba ^{-1} $ can be shortened, whereas $ aca ^{-1} $ cannot. Consider, for example, a word of the form $ acbscba ^{-1} $ for some $ s $ chosen randomly from $ \{ a ^{\pm},b ^{\pm},c ^{\pm}  \} $. We want to quantify the probability with which the word can be shortened. This is the role of pilings in our argument. 

We start off by extending the definition of pilings to graph products. We explain how to produce a piling for a word in $ \sqcup _{i=1} ^{n} G _{i} $, then show that this is independent of the choice of word representative. This will produce a well-defined piling for an element of $ G $. 

Let $G$ be a graph product with vertex groups $G _{1}, \dots, G _{D}$. Let $\mathcal{A} = \{0 \} \sqcup _{i=1} ^{D} G _{i} \setminus \{ e \}$, 
and let $\mathcal{A}^\star$ the set of finite words in the alphabet $\mathcal{A}$. We denote as $\epsilon$ the empty word.
A piling for $G$ is a map $\Pi: G \to (\mathcal{A}^*)^D$ defined as follows.

\begin{defin}
Let $ G $ be a graph product with vertex groups $ G _{1}, \dots, G _{D} $. A \emph{piling} $ \Pi(h) $ for an element $ h \in G $ is a list of $ D $ words in the alphabet $\mathcal{A}$, defined inductively as follows:
	\begin{itemize} 
		\item The piling $ \Pi(e) $ for the trivial element is empty.
		\item If $ h = h'g _{i} $, where $ g _{i} \in G _{i} $, then 
			\begin{enumerate} 
				\item If the $i$th string is empty or ends in a $ 0 $, then the $i $th string of $ \Pi(h) $ is given by appending $ g _{i} $ to the $ i $th string of $ \Pi(h') $, and a $ 0 $ to the $ j $th string for every $ j $ such that vertices $ v _{i} $ and $ v _{j} $ are not adjacent.
				\item If the $ i $th string of $ \Pi(h') $ ends in an element $ g' _{i} $ of $ G _{i} $, then the $ i $th string of $ \Pi(h) $ is $ g' _{i} g _{i} $. If $ g' _{i} g _{i} $ is the identity, then remove the $ 0 $ on the $ j $th strings, where $ v _{i} $ and $ v _{j} $ are not adjacent.
				  \end{enumerate}
	 \end{itemize}	 
\end{defin} 
	
By a quick computation, one can see that if $ s _{i} $ and $ s _{j} $ are elements of the adjacent groups $ G _{i} $ and $ G _{j} $, then $ \Pi(hs _{i} s _{j}) = \Pi(h s_{j} s _{i}) $. Likewise, we have $ \Pi(hs _{i} s ^{-1} _{i}) = \Pi(h) $. By \cite[Lemma 3.1]{HM}, this means we have a well-defined (independent of word representative) piling for any element of $ G $.

\medskip

	 \textbf{Example 1:} Consider the group $ \mathbb{Z} ^{2} * \mathbb{Z} = \langle a,b,c | [a, b] \rangle $. Then \begin{itemize}
		 \item A piling for $ a $ is $ \left(a,  \epsilon, 0\right) $.
		 \item A piling for $ ac$ is $ \left(a0, 0, 0c\right) $.
		 \item A piling for $ acb$ is $ \left(a0,0b,0c0\right) $. 

	 \end{itemize} 
	 \begin{figure}[!h]
		 \centering
\tikzset{every picture/.style={line width=0.75pt}} 

\tikzset{every picture/.style={line width=0.75pt}} 

\begin{tikzpicture}[x=0.75pt,y=0.75pt,yscale=-1,xscale=1]

\draw    (140,200) -- (500.14,200) ;
\draw    (230.92,31) -- (230.92,200) ;
\draw    (330.92,30) -- (330.92,200) ;
\draw    (430.92,31) -- (430.92,200) ;

\draw (280,180) node [anchor=south west][inner sep=0.75pt]  [font=\Huge]  {$0$};
\draw (180,130) node [anchor=south west][inner sep=0.75pt]  [font=\Huge]  {$0$};
\draw (380,180) node [anchor=south west][inner sep=0.75pt]  [font=\Huge]  {$0$};
\draw (380,80) node [anchor=south west][inner sep=0.75pt]  [font=\Huge]  {$0$};
\draw (180,180) node [anchor=south west][inner sep=0.75pt]  [font=\Huge]  {$a$};
\draw (280,130) node [anchor=south west][inner sep=0.75pt]  [font=\Huge]  {$b$};
\draw (380,130) node [anchor=south west][inner sep=0.75pt]  [font=\Huge]  {$c$};

\end{tikzpicture}

		 \caption{A piling for $ acb $.}
	  \end{figure}

	 Observe that the piling for $ \Pi(g ^{-1}) $ is given by reversing all strings in $ \Pi(g) $ and swapping each $ g _{i}$ with $ g _{i} ^{-1} $.

	 \begin{defin} 
		 Given an element $ g \in G $, the terminal clique $ \term(g) $ is the set of vertices $ v _{i} $ in the graph such that the $ i $th string in the piling $ \Pi(g) $ ends in a nontrivial element of $ G _{i} $.
	  \end{defin}

	  Likewise we define the \emph{initial clique} $ \init(g) $ to be the set of vertices $ v _{i} $ such that the $ i $th string of $ \Pi(g) $ starts with a nontrivial element of $ g _{i} $. Observe that $ \init(g) = \term(g ^{-1}) $.

	  \begin{lem} 
	Let $ g,h$ be elements in the graph product $ G $ such that $ \term(g) \cap \init(h) = \varnothing $. Then $ \Pi(gh) = \Pi(g)\Pi(h) $. 
	  \end{lem}
	  \begin{proof}
		  Given an element $ h \in G $, write $h = w _{1} w _{2} \dots w _{n} $ where each $ w _{i}  $ lies in a different vertex group. Let the \emph{syllable length} of $ h $ be the length $ n $ of the shortest such product. We induct on the syllable length of $ h $.

		  In our base case, $ h $ is just an element from one vertex group $ G _{i} $, so that $ \init(h) = \{ v _{i} \} $. As this set is disjoint from $ \term(g) $, we know that the piling for $ g$ ends with a 0 in the $ i $th string, or the $i$th string is empty. In either case, the piling $ \Pi(gh) $ is given by concatenating $ \Pi(g), \Pi(h) $.

		  Now suppose that the syllable length for $ h $ is equal to $ n > 1$, and let $h =  w _{1} \dots w _{n} $ be a representation of minimal length, which must be reduced. Let $ v $ be the vertex corresponding to $ w _{1} $. Then if $ w _{1} $ and $ w _{k} $ lie in the same vertex group, then there exists $ 1 < i < k $ such that $ v $ and $ v _{i} $ are not adjacent. Then the piling for $ h $, in some coordinate, begins with $ w _{1}0  $. As a result, the piling for $ w _{1} ^{-1} h $ begins with a $ 0 $ in the coordinate corresponding to $ w _{1} $, so that the initial clique of $ w _{1} ^{-1} h $ is disjoint from $ \{ v \} $. Applying our induction hypothesis, we have \begin{align*} 
			  \Pi(g)\Pi(h) &= \Pi(g) \Pi(w _{1}w _{1} ^{-1}h) \\
				       &= \Pi(g) \Pi(w _{1}) \Pi(w _{1} ^{-1} h) \\
				       &= \Pi(g w _{1}) \Pi(w _{1} ^{-1} h) \\
				       &= \Pi(gh)
			    \end{align*}
	  \end{proof}

	  \begin{defin} 
		    To a word $ \pi \in \mathcal{A}^{*} $, let the syllable length of $ \pi $ be the number of nonzero letters. To a piling $ \Pi = (\pi_{1}, ..., \pi _{D})  $, we define the syllable length of $ \Pi $ be the sum of syllable lengths of $ \pi _{i} $ for $ 1 \leq i \leq n $.
		    \end{defin}


\section{Pivotal Points}

	Our plan is to control the behaviour of a sample path by considering times where it goes in independent directions. In the free group, there are points where the sample path lies in some subtree forever after some time $ n $. Intuitively, we should be able to pivot a sample path about such a point, and get another sample path with the same drift. We plan to show that there are many such points, and that for each point there are many directions in which the sample path moves further away from the identity.

Recall that the paths for our walks are of the form $Z_n = s_1 w_1 \dots s_n w_n$. 

\begin{defin} 
Given a finite path $(s_1, w_1, \dots, s_n, w_n)$ 
of length $2 n$, a time $ k < n$ is \emph{pivotal} with respect to $ n $ if: 
	\begin{itemize} 
		\item The piling of $ Z _{k-1} s _{k} $ prefixes the pilings for $ Z _{k}, Z _{k} s_{k+1}, Z _{k+1}, \dots, Z _{n}  $.
		\item The element $ s _{k} $ is not in $ \term(Z _{k-1}) $ and $ \term(Z _{k-1} s _{k}) $ is disjoint from $ \init(w _{k}) $.
	 \end{itemize}
\end{defin}

	 In the case of the free group, the first condition states that the sample path lies in the subtree starting at $ Z _{k-1} s _{k} $ for all times between $ k $ and $ n $. If the second bullet point is met, we say that $ s _{k} $ satisfies the \emph{local geodesic condition}. To motivate this name, observe that the requirement $ s _{k} \notin \term(Z _{k-1}) $ implies, that $ |Z _{k-1}| < |Z _{k-1} s _{k}| $, where this norm refers to the syllable length. Also, the requirement $ \term(Z _{k-1} s _{k}) \cap \init(w _{k}) = \varnothing$ implies that $ |Z _{k-1} s _{k}| < |Z _{k-1} s _{k} w _{k}| $. Hence at time $ k $ our sample path moves strictly further away from the identity. 

	  Since the piling $ \Pi(Z _{k-1} s _{k}) $ is a prefix for the rest of the sample path, then the syllable length of $ \Pi(Z _{k-1} s_{k})$ lower bounds the syllable length of each of $ \Pi(Z _{k}), \Pi(Z _{k} s_{k+1}),\dots, \Pi(Z _{n})$. 
		  
	  Now if $ k _{1}, \dots, k _{|P _{n}|} $ are our pivotal times, then by the above we get \[0 \leq |Z _{k _{1}-1}| < |Z _{k _{1}+1} s _{k _{1}}| \leq |Z _{k _{2}-1}| < |Z _{k _{2}-1} s _{k _{2}}| \leq \dots \leq|Z _{k _{|P _{n}|} - 1}| < |Z _{k _{|P _{n}| - 1} } s_{|P _{n}|} | \leq |Z _{n}|.\]

	  The above discussion can be summarized in the following lemma:

\begin{lem}
	The distance travelled $ |Z _{n}| $ is bounded below by the number of pivotal points $ \# P _{n} $.
\end{lem}

		    As we are conditioning on the words $ w _{1}, \dots, w _{n} $ coming from the measure $ \nu $, our random walk is determined by the words $ s _{1}, \dots, s _{n} $ coming from $ \mu $. Hence to understand our random walk, we should consider the sequence $ (s _{1}, \dots, s _{n}) $ associated to some sample path. 
		    \begin{defin} 
			        A sequence $ (s _{1}, \dots, s _{n}) $ is \emph{pivoted} from $ (s' _{1}, \dots, s' _{n}) $ if they have the same pivotal times and $ s _{k} = s' _{k} $ for each time $ k $ which is not pivotal.
			      \end{defin}
			      The idea here is that two pivoted sequences give rise to sample paths which are in some sense equivalent in terms of their drift. Indeed, this induces an equivalence relation on our set of sequences. Given some sequence $ \overline{s} = (s _{1}, ..., s _{n}) $, let $ \mathcal{E} (\overline{s}) $ be its equivalence class with respect to this relation.

			     We claim that these equivalence classes are large. That is to say, any sequence has many pivoted sequences. This is the key lemma, which will allow us to bound the number of pivotal times from below by a sum of i.i.d. independent random variables.

		    \begin{lem} 
			    Let $ k $ be a pivotal time and replace $ s _{k} $ with any $ s' _{k} $ such that $ s' _{k} $ is not in $ \term(Z _{k-1}) $ and $ s'_{k} $ is not adjacent to $ \init(w _{k}) $. Then the sequence $ (s _{1}, \dots, s _{k}, \dots, s _{n}) $ is pivoted from the sequence $ (s _{1}, \dots, s' _{k}, \dots, s _{n}) $.
		    \end{lem} 
		    \begin{proof}
			    We need to show that all pivotal times remain pivotal.

		    First let $ j < k $ be a pivotal time before $ k $. We still know that $ s _{j} $ is not in $ \term(Z _{j-1}) $ and $ \term(Z _{j-1} s _{j}) \cap \init(w _{j}) = \varnothing $, so we need to justify that $ Z _{j-1} s _{j} $ is still a piling-prefix for $ Z _{j+1}, Z _{j+1}s _{j+2}, \dots, Z _{n} $. We already have that the piling of $ Z _{j-1} s _{j} $ prefixes the piling of everything in this sequence up to $ Z _{k-1} $, and since $ s _{k} $ satisfies the local geodesic condition we have that $ \Pi(Z _{j}) = \Pi(Z _{k-1})\Pi(s _{k} w _{k}\dots s _{j}w _{j}) $ for all $ k \leq j \leq n $, so that $ j $ is still pivotal.

		    Now we claim that $ k $ is a pivotal time. We only need to show that $ Z _{k-1} s _{k} $ prefixes the rest of the sequence. As $ s' _{k} $ is not adjacent to $ \init(w _{k}) $, then we know that $ \term(Z _{k-1} s' _{k})\cap \init(w _{k}s _{k+1}\dots w _{j})  $ for all $ j > k$. Hence by our lemma we know that the piling $ \Pi(Z _{j})$ is equal to $\Pi(Z _{k-1} s' _{k}) \Pi(w _{k}s _{k+1} \dots w _{j})  $. By the same argument, replacing $ Z _{j} $ with $ Z _{j} s _{j+1} $, we have that $ Z _{k-1} s _{k} $ is a prefix to everything else in the sequence.

		    Finally let $ j > k $ be a pivotal time after $ k $. We already know that $ Z _{j-1} s _{j} $ prefixes the rest of the sequence after time $ j $. As $ s' _{k}  $ is not adjacent to $ \init(w _{k}) $, then $ \term(Z _{k-1} s _{k}) $ is disjoint from $ \init(w _{k}) $. Hence the terminal clique $ \term(Z _{j-1}) $ does not contain $ s _{j} $. As $ s _{j} $ is not adjacent to $ \init( w _{j}) $ then $ \term(Z _{j-1} s _{j}) $ is disjoint from $ \init(w _{j}) $, so that $ j $ is still pivotal.
	    \end{proof}

	    \section{Main Argument}
	    In this section we prove Theorem 1. First we recall some notation. Let $ G $ be a graph product with $ D $ vertices and vertex groups $ G _{1}, \dots, G _{D} $. Let $ C $ be the size of the largest clique, and $ B $ the maximum size of the 1-neighbourhood $N _{1} (K) $ where $ K \subset \Gamma$ ranges through all cliques. Let $ \mu $ be a measure on $ G $ such that $ \mu (G _{i} \setminus \{ e \} ) = \frac{1}{D} $. In other words, $ \mu $ is equally likely to pick out a nontrivial element of any group. Further suppose that $ \nu $ is a measure with $ \nu (e) = 0 $. Consider a random walk driven by $ \mu * \nu $.

	    We want to show that there exists some $ \kappa > 0 $ such that \[ \mathbb{P} (|Z _{n}| \leq \kappa n) \leq e ^{-\kappa n},\] where $ |Z _{n}| $ denotes the word length of $ Z _{n} $ with respect to some fixed generating set. Write $ Z _{n} = s _{1} w _{1}...s _{n} w _{n} $ where $ s _{n} \sim \mu $ and $ w _{n} \sim \nu $. We condition on the $ w _{n} $'s and keep the randomness coming from the $ s _{n}$'s. Then we find a working $ \kappa $ that is independent of our conditioning. To this end, we use our assumptions on the graph to show that the drift is bounded from below by a sum of i.i.d. variables with positive expectation.  

	    Let $ k _{1} < \dots < k _{m} $ be the pivotal times for the sequence $ \overline{s} = (s _{1}, ..., s _{n})$. Conditioning on the equivalence class $ \mathcal{E} (\overline{s}) $, as defined in the previous section.
	     We know from the previous lemma that the random variables $ s _{k _{i}} $ are all independent. We use this independence to argue that there are many pivotal points. 

	    \begin{lem}\label{KeyLemma} 
		    Let $ A _{n} = \# P _{n} $, and suppose that $0 < \frac{B}{D-(B+ C)} < 1 $. Also let $ U $ be an integer-valued random variable, independent of $ A _{n} $ and with distribution  \[ \mathbb{P} (U = 1) = \frac{D-(B+ C)}{D }, \] and \[ \mathbb{P} (U \leq -j) = \frac{B+ C}{D }\cdot \left(\frac{B}{D-(B+ C)} \right) ^{j-1} \] for all $ j > 0 $. Then $ A _{n+1} $ stochastically dominates $ A _{n} +U $ in the sense that 
		    $$ \mathbb{P}(A _{n +1} \geq i) \geq \mathbb{P}(A _{n} + U \geq i)   \qquad \textup{for all } i.$$
		     \end{lem}
		      \begin{proof}
			      Fix a sequence $ \overline{s} =(s _{1}, \dots, s _{n}) $ and condition on $ \mathcal{E} (\overline{s}) $. Let $ \overline{s} ' \in \mathcal{E} (\overline{s}) $.

			      Now consider the probability that $ A_{n+1} = A_{n} +1 $. After the last pivotal time, the local behaviour of the random walk is the same over the entire equivalence class $ \mathcal{E} ( \overline{s}) $ - that is to say, the terminal clique of $ Z _{n}' $ is constant over the equivalence class. 

			      We know that the terminal clique $ \term(Z' _{n}) $ has at most $ C $ vertices. Hence there are at least $ D - C $ vertices $ v _{i} $ which are disjoint from $ \term(Z _{n}') $. Likewise, we know that there are at most $ B  $ vertices contained in or adjacent to $ \init(w _{n+1})  $, therefore there are at least $ D - B $ choices for $ G _{i _{n+1}} $ such that $ s _{n+1} $ is not adjacent to $ w _{n+1} $. Hence there are at least $ D-(B+ C) $ choices for $ s _{n+1} $ which add another pivotal point. As our probability distribution is uniform over our set of groups then \[ \mathbb{P} (A _{n+1} \geq A _{n} +1 | \mathcal{E}(\overline{s})) \geq  \frac{D-(B+ C)}{D} .\]

			      Now fix $ j > 0 $ and consider the probability that $ A _{n+1} \leq A _{n} - j $. This first requires that $ s _{n+1} $ fails the local geodesic condition, which happens with probability $ \frac{B+C}{D} $. For $ k _{q} $ to no longer be pivotal, this requires that $ s' _{k _{q}} $ is now adjacent to the initial clique of $ w _{k _{q}}\dots s _{n+1} w _{n+1} $, which happens with probability at most $ \frac{B}{D-( B + C)} $. Conditioned on $ \mathcal{E}  ( \overline{s}) $, the pivotal times $ s _{n+1} $ and $ s' _{k _{q}} $ are independent. Hence \[ \mathbb{P} (A _{n+1} \leq A _{n}-1 | \mathcal{E}  (\overline{s})) \leq \frac{B + C}{D} \cdot \frac{B}{D-( B+ C)} .\]
			      By the same argument, as the $ s' _{k _{i}}  $ are independent, we have that 
			      \[ \mathbb{P} (A _{n +1} \leq A _{n} - j| \mathcal{E} ( \overline{s})) \leq \frac{B + C}{D} \cdot \left(\frac{B}{D-( B+ C)} \right)^{j-1} \] 
for all $ j > 0 $. 

			      As this bound is uniform over conditioning, we have the conclusion of the lemma.
\end{proof}

		      We also make use of the following lemma, standard in the theory of large deviations. 

 \begin{lem} 
	 Let $ U _{1}, \dots, U _{n} $ be i.i.d. copies of a random variable $ U $ and let $ t >0 $ be such that $ \mathbb{E} [e ^{-tU}] < 1 $. Then for $ \kappa < \frac{- \ln \mathbb{E} [e ^{-tU}]}{1+t} $ we have \[ \mathbb{P} (U _{1} + \dots + U _{n} \leq \kappa n) \leq e ^{-\kappa n} .\]  
		    \end{lem}
		    \begin{proof}
			   By Markov's inequality and independence we have, for any $ \kappa > 0 $, 
			   \[ \mathbb{P} (U _{1} + \dots + U _{n} \leq \kappa n) \leq e ^{t\kappa n} \left(   \mathbb{E} \left[e ^{-t U}\right] \right)^n .\] 
			   If we pick 
			   \[ \kappa < \frac{- \ln \mathbb{E} [e ^{-t U}]}{1+t},\] 
			   then the right hand side is less than $ e ^{-\kappa n} $.
		    \end{proof}

		    Now we are ready to prove our main theorem.

\begin{proof}
We have 
\[ \mathbb{E}[U] = \frac{D-B-C}{D} - \frac{(B+C)}{D} \left( \frac{D-B-C}{D-2B-C}\right) .\]

If $ D > 3B+2C $, then $ \mathbb{E} [U] > 0 $. Then $\frac{d}{dt}| _{t = 0 ^{+}} \mathbb{E} [e ^{-tU}] < 0 $, so there exists some $ t > 0 $ such that $ \mathbb{E} [e ^{-tU}] < 1 $. Then there exists some positive $\kappa$ with 
		      \[ \kappa < \frac{- \ln \mathbb{E} [e ^{-tU}] }{1+t}. \]  
		      Now let $ U _{1}, \dots, U _{n} $ be $ n $ i.i.d. copies of $ U $. Iterating the previous lemma we get that $ A _{n} $ stochastically dominates $ U _{1} + \dots + U _{n} $. Hence by our large deviations bound we have \[ \mathbb{P} (|Z _{n}| \leq \kappa n) \leq \mathbb{P} (A _{n} \leq \kappa n) \leq \mathbb{P} (U _{1} + \dots + U _{n} \leq \kappa n) \leq e ^{- \kappa n} . \]
		    \end{proof}

Given a graph product of groups, one can compute the drift for a random walk induced by an alternating measure as follows:

\begin{enumerate} 
	\item Verify that $ D > 3B + 2C $.
	\item Maximize the quantity \[ \frac{- \ln \mathbb{E}[ e ^{-tU}] }{1+t}\] with the constraints $ t > 0 $, $ \mathbb{E} [e ^{-tU}] < 1 $.
	  \end{enumerate}

	  For an example, we consider the family of graphs which are cycles of length $ D $. In this case we have $ C = 2 $ and $ B = 4 $, so that the theorem applies for $ D > 14 $. We compute the drift afforded from Theorem 1 for $ 15 \leq D \leq 12000 $, shown in figure 2.

\begin{figure}[!h] \label{CycleDrift}
	 \centering
	 \includegraphics[width=\textwidth]{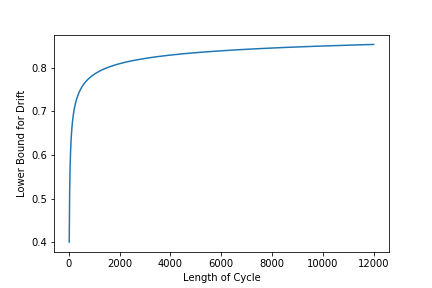}
	 \caption{Our drift estimate for an alternating random walk on a graph product given by a $ D $-cycle.}
 \end{figure}

With further assumptions on $ B,C, $ and $ D $ one can derive asymptotics for this drift estimate. 

\begin{prop} 
	Let $ \alpha \in (0,1/2) $ and suppose that $ B,C \leq o(D ^{1-2\alpha}) $ as $ D \to \infty $. Let $ T $ be the set of positive $ t $ such that $ \mathbb{E} [e ^{-tU}] < 1 $. Then 
	\[ \sup _{t \in T} \frac{- \ln \mathbb{E}\left[ e ^{-tU}\right] }{1+t} \to 1\] 
as $ D \to \infty $. 

	  \end{prop}
	  \begin{proof}
		  Since $ - \ln $ is convex, by Jensen's inequality we have \[ \frac{-\ln \mathbb{E} \left[e ^{-tU}\right]}{1+t} \leq \frac{\mathbb{E} \left[- \ln e ^{-tU}\right] }{1+t} = \frac{t \mathbb{E} U}{1+t} .\] Since $ \mathbb{E} U > 0 $, then this final term is increasing in $ t $ and goes to $ \mathbb{E} U $ as $ t \to \infty $. Since $ B,C \leq o(D) $ then $ \mathbb{E} U \to 1 $ as $ D \to \infty $. Therefore \[ \lim_{D \to \infty} \sup _{t \in T} \frac{-\ln \mathbb{E} \left[e ^{-tU}\right]}{1+t} \leq 1 .\]

		  Pick $ t _{D} = \ln D ^{\alpha} $. As $ B, C \leq o(D) $, then $ t_D $ satisfies $ \mathbb{E} \left[ e ^{-t_DU}\right] < 1 $ for large enough $ D $.

	  Then \[ \mathbb{E} \left[e ^{-t _{D} U}\right] = D ^{-\alpha} \frac{D-B-C}{D} + (B+C) D ^{\alpha - 1} \frac{D-2B-C}{D-B-C-D ^{\alpha} B}   .\]

	  Therefore \begin{align*} 
		  \frac{-\ln \mathbb{E} \left[e ^{-t_D U} \right]}{1+t_D} &= \frac{-\ln\left( D ^{-\alpha} \frac{D-B-C}{D} + (B+C) D ^{\alpha-1} \frac{D-2B-C}{D-B-C-D ^{\alpha} B}\right)}{1+\ln D ^{\alpha}} \\
									  &\sim \frac{-\ln D ^{-\alpha}}{1+\ln D ^{\alpha}} \\
									  &= \frac{\ln D ^{\alpha}}{1+\ln D ^{\alpha}} \\
									  &\to 1,
		    \end{align*}
		    as $ D \to \infty $, where the second line comes from the fact that $ B, C \leq o(D ^{1-2\alpha}) $ and that the second summand is of order strictly less than $ D ^{1-2\alpha} D ^{\alpha -1} = D ^{\alpha} $.
	  \end{proof}

\newpage
\bibliographystyle{alpha}
\bibliography{GraphProductDrift.bib}

\begin{thebibliography}{BMSS20}

\bibitem[BMSS20]{boulanger2020large}
Adrien Boulanger, Pierre Mathieu, Cagri Sert, and Alessandro Sisto.
\newblock Large deviations for random walks on hyperbolic spaces.
\newblock {\em arXiv preprint arXiv:2008.02709}, 2020.

\bibitem[BQ16]{benoist2016central}
Yves Benoist and Jean-Fran{\c{c}}ois Quint.
\newblock Central limit theorem on hyperbolic groups.
\newblock {\em Izvestiya: Mathematics}, 80(1):3--23, 2016.

\bibitem[CGW08]{crisp}
John Crisp, Eddy Godelle, and Bert Wiest.
\newblock The conjugacy problem in right-angled artin groups and their
  subgroups.
\newblock {\em arXiv preprint arXiv:0802.1771}, 2008.

\bibitem[Cho21a]{choi2021central}
Inhyeok Choi.
\newblock Central limit theorem and geodesic tracking on hyperbolic spaces and
  teichm\" uller spaces.
\newblock {\em arXiv preprint arXiv:2106.13017}, 2021.

\bibitem[Cho21b]{choi2021pseudo}
Inhyeok Choi.
\newblock Pseudo-anosovs are exponentially generic in mapping class groups.
\newblock {\em arXiv preprint arXiv:2110.06678}, 2021.

\bibitem[Fur63]{furstenberg1963poisson}
Harry Furstenberg.
\newblock A poisson formula for semi-simple lie groups.
\newblock {\em Annals of Mathematics}, pages 335--386, 1963.

\bibitem[Gou21]{gouezel2021}
S{\'e}bastien Gou{\"e}zel.
\newblock Exponential bounds for random walks on hyperbolic spaces without
  moment conditions.
\newblock {\em arXiv preprint arXiv:2102.01408}, 2021.

\bibitem[HM95]{HM}
Susan Hermiller and John Meier.
\newblock Algorithms and geometry for graph products of groups.
\newblock 1995.

\bibitem[Kai00]{kaimanovich2000poisson}
Vadim~A Kaimanovich.
\newblock The poisson formula for groups with hyperbolic properties.
\newblock {\em Annals of Mathematics}, pages 659--692, 2000.

\bibitem[KM99]{karlsson1999multiplicative}
Anders Karlsson and Gregory~A Margulis.
\newblock A multiplicative ergodic theorem and nonpositively curved spaces.
\newblock {\em Communications in mathematical physics}, 208(1):107--123, 1999.

\bibitem[MT18]{maher2018random}
Joseph Maher and Giulio Tiozzo.
\newblock Random walks on weakly hyperbolic groups.
\newblock {\em Journal f{\"u}r die reine und angewandte Mathematik (Crelles
  Journal)}, 2018(742):187--239, 2018.

\bibitem[NS13]{nevo2013poisson}
Amos Nevo and Michah Sageev.
\newblock The poisson boundary of cat(0) cube complex groups.
\newblock {\em Groups, Geometry, and Dynamics}, 7(3):653--695, 2013.

\bibitem[Sun20]{sunderland2020linear}
Matthew~H Sunderland.
\newblock Linear progress with exponential decay in weakly hyperbolic groups.
\newblock {\em Groups, Geometry, and Dynamics}, 14(2):539--566, 2020.

\bibitem[Tio15]{tiozzo2015sublinear}
Giulio Tiozzo.
\newblock Sublinear deviation between geodesics and sample paths.
\newblock {\em Duke Mathematical Journal}, 164(3):511--539, 2015.

\end{thebibliography}

\end{document}